\documentclass[12pt]{article}

\usepackage{cite}

\usepackage{hyperref} 
\usepackage[utf8]{inputenc}
\usepackage[T1]{fontenc}
\usepackage{lmodern}

\usepackage[english]{babel}

\usepackage{todonotes}
\usepackage{verbatim} 
\usepackage{enumerate}
\usepackage{mathtools}
\usepackage[normalem]{ulem}
\usepackage{bbm}

\usepackage{graphicx}
\usepackage{amsmath, amssymb,amsthm}
\usepackage[arrow, matrix, curve]{xy}
\numberwithin{equation}{section}

\newtheorem{theorem}{Theorem}[section]
\newtheorem{lemma}[theorem]{Lemma}
\newtheorem{proposition}[theorem]{Proposition}

\theoremstyle{definition}
\newtheorem{definition}[theorem]{Definition}

\newtheorem{remark}[theorem]{Remark}

\newcommand{\cA}{\mathcal{A}}

\newcommand{\cM}{\mathcal{M}}

\newcommand{\cP}{\mathcal{P}}

\newcommand{\cX}{\mathcal{X}}
\newcommand{\cY}{\mathcal{Y}}

\newcommand{\bN}{\mathbb{N}}

\newcommand{\bR}{\mathbb{R}}

\newcommand{\bT}{\mathbb{T}}

\newcommand{\bZ}{\mathbb{Z}}

\newcommand{\ip}[2]{\langle #1,#2\rangle}
\newcommand{\dd}{ \mathrm{d}}
\newcommand{\bONE}{\mathbbm{1}}
\newcommand{\vn}[1]{\left| \! \left| #1\right| \! \right|} 

\DeclareMathOperator*{\argmin}{arg\,min}
\DeclareMathOperator*{\esssup}{ess\,sup}

\DeclareMathOperator{\arctanh}{arctanh}

\def\N{{\mathbb N}}

\def\D{\Delta}

\def\b{\beta}

\def\d{\delta}

\def\phi{\varphi}

\def\l{\lambda}

\def\r{\rho}

\def\s{\sigma}

\def\D{\Delta}

\def\L{\Lambda}

\def\T{\T}

\newcommand{\footnoteremember}[2]{
	\footnote{#2}
	\newcounter{#1}
	\setcounter{#1}{\value{footnote}}
}
\newcommand{\footnoterecall}[1]{
	\footnotemark[\value{#1}]
}

\begin{document}
\title{Gibbs-non-Gibbs transition in the fuzzy Potts models with a Kac-type interaction: Closing the Ising gap}

\author{
Florian Henning,\footnoteremember{RUB}{Ruhr-Universit\"at   Bochum, Fakult\"at f\"ur Mathematik, D44801 Bochum, Germany} \footnote{Florian.Henning@rub.de}
\and
Richard C. Kraaij,\footnoterecall{RUB} \footnote{Richard.Kraaij@rub.de}
\and Christof K\"ulske \footnoterecall{RUB} \footnote{Christof.Kuelske@ruhr-uni-bochum.de, 
\newline
\texttt{http://www.ruhr-uni-bochum.de/ffm/Lehrsttuehle/Kuelske/kuelske.html
/$\sim$kuelske/ }}\, 
\,  
}

\maketitle

\begin{abstract}
	We complete the investigation of the Gibbs properties of the fuzzy Potts model on the $d$-dimensional torus with Kac interaction which was started by Jahnel and one of the authors in \cite{JaKu17}. As our main result of the present paper, we extend the previous sharpness result of mean-field bounds to cover all possible cases of fuzzy transformations, allowing also for the occurrence of Ising classes. The closing of this previously left open Ising-gap involves an analytical argument showing uniqueness of minimizing profiles for certain 
	non-homogeneous conditional variational problems. 
\end{abstract}

\smallskip
\noindent {\bf AMS 2000 subject classification:} 82B20, 82B26.

 \smallskip
\noindent {\bf Keywords:}  Potts model, Kac model, fuzzy Kac-Potts model, Gibbs versus non-Gibbs, large deviation principles, diluted large deviation principles.  


\vfill\eject

\section{Introduction}

The Gibbs property is an important regularity property of a large, or infinite system which comes in various versions, according to the setting considered. It can be formulated for lattice models in the infinite volume \cite{EFS93,Ge11}, for systems of point particles in euclidean space \cite{Ru99} for mean-field systems\cite{KN07}, or for Kac-systems \cite{FHM14}. 

Often local transformations applied to realizations of a given large system which behaves according to a nice Gibbs distribution are of interest. Such transformations include discretization transformations, and various sampling procedures, which reduce information of the initial data,  cf. \cite{EFS93,HK04,JaKu17}. Stochastic time-evolutions, motivated from physics, provide another very interesting type of such transformations, cf. \cite{KN07,EK10,FHM13,FHM14,HRZ15,JaKu16,JaKu17a}.


A relevant and natural question in this context is whether it is possible to describe the image system again as a system of a nice Gibbsian form, with effective, new interactions. This may be true (or not true) for reasons having to do with the absence (or the existence, and visibility) of internal phase transitions. These internal phase transitions may act as a switch and cause an infinite range dependence on small variations of a conditioning. If we want to perform a concrete analysis to decide if such a mechanism shows up for a given system, given parameter values, and given transformation, both mean-field systems and Kac-systems are hopeful as they have a large parameter, which allows for an asymptotic description in terms of a variational principle.  

Indeed, after some work, one sees that also Gibbs-non-Gibbs(GnG) transitions in the Kac-limit, are tied to {\em conditional} variational problems for profiles, via large deviations.  There is also a natural link from Kac-models to mean-field models on the level of these variational principles, which one obtains by considering only profiles which do not have any spatial dependence. 


In the present paper, we round up the previous investigation of the GnG transitions in the Kac fuzzy Potts model of \cite{JaKu17}, and fully decide on an open case which was left out, as we will explain now. The Kac fuzzy Potts model is defined as the image of the Kac-version of Potts model, under the deterministic local transformation which distinguishes the possible spin values of the Potts local spin space $\{1, 2, . . . , q\}$ only according to a fixed local partitioning into $r$ subclasses. The Kac-Potts model itself is formulated on a d-dimensional torus, a setup similar to that of \cite{FHM14} in their study of the dynamical Kac-Ising model. We are now able to completely answer the question:  

{\em Are the non-Gibbs parameter regions equal in Kac-setup and mean-field setup?} 

Or, putting this question in a slightly more refined way: Does the possibility for spatial inhomogeneity in the Kac-model create new and "worse" bad configurations which are seen in parameter regions beyond the bad mean-field regions? It turns out that worse bad configurations are indeed not possible,  but for a non-trivial reason involving proper treatment of a non-homogeneous conditional variational problem, arising for the Ising-classes. 
This progress made in the present paper relies essentially on the new analytical uniqueness of minimizers-result of Proposition \ref{proposition:minimizers_I_and_convergence}, which leads to the proof of the main result Theorem \ref{theorem:fuzz_kac_potts_gng}.

The paper is organized as follows. We keep the presentation self-contained, at the same time streamlining and shortening some parts of the arguments. In Section \ref{section:model_main_results} we introduce our model and state the main results. In particular, in Sections \ref{section:Kac_Potts_model} to \ref{section:seq_gibbs_for_FKP} we introduce the Potts-Kac, fuzzy Potts-Kac model, the notion of sequential Gibbsianness and our main result, Theorem \ref{theorem:fuzz_kac_potts_gng} on Gibbs-non-Gibbs transitions. In Section \ref{section:representation_kernels_and_LDP}, we recall a representation of the conditional kernels that play an important role in the study of GnG transitions and recall the diluted large deviation principle from \cite{JaKu17}. In Section \ref{section:Intro_minimizers_and_limiting_kernels}, we give the main technical result, Proposition \ref{proposition:minimizers_I_and_convergence}, and conclude with a representation of the limiting conditional probability kernels in the Gibbs case, given in Theorem 
\ref{theorem:form_of_kernels_in_Gibbs_case}. The result shows that Gibbsianness of a transformed system
may very well hold in the presence of internal phase transitions, reflected here as the functional dependence of limiting kernels on 
non-trivial minimizing profiles which nevertheless themselves behave continuously as a function of conditioning profiles. 

The proofs follow in Section \ref{section:proofs} and we conclude with two appendices on continuity of limiting functionals and convergence of random variables in the setting where an LDP holds with a finite number of minimizers.

\smallskip

\textbf{Acknowledgement} RK is supported by the Deutsche Forschungsgemeinschaft (DFG) via RTG 2131 High-dimensional Phenomena in Probability--Fluc\-tuations and Discontinuity.

\section{Model and main results} \label{section:model_main_results}
\subsection{The Kac-Potts model} \label{section:Kac_Potts_model}

Let $\mathbb{T}^d :=\bR^d/\bZ^d$ be the $d$-dimensional unit torus. For $n\in \N$, let $\mathbb{T}^d_n$ be the $(1/n)$-discretization of $\mathbb{T}^d$ defined by $\mathbb{T}^d_n:=\Delta^d_n/n$, with $\Delta^d_n:=\bZ^d /n\bZ^d$ the discrete torus of size $n$. For $n\in\N$, let $\Omega_{q,n}:=\{1,\dots,q\}^{\Delta^d_n}$ be the set of \textit{Potts-spin configurations} on $\Delta^d_n$. We will call elements of $\{1,\dots,q\}$ \textit{colours}. The energy of the configuration $\sigma := (\sigma(x))_{x\in\Delta^d_n}\in\Omega_{q,n}$ is given by the \textit{Kac-type Hamiltonian} 
\begin{equation}
H_n(\s):=-\frac{1}{n^d}\sum_{x,y\in\D^d_n}J\left(\frac{x-y}{n}\right)1_{\s(x)=\s(y)}, \hspace{1cm}\sigma \in \Omega_{q,n}
\end{equation}
where $0 < J\in C(\mathbb{T}^d)$ is a continuous \textit{interaction-functions} on $\mathbb{T}^d$ which is symmetric and $\int \dd v J(v)=1$. The Gibbs measure associated with $H_n$ is given by
\begin{equation}\label{KPM}
\mu_n(\s):=\frac{1}{Z_n}\exp(-\b H_n(\s)),\hspace{1cm}\s\in\Omega_{q,n}
\end{equation}
with $\beta \in[0,\infty)$ the inverse temperature and $Z_n$ the normalizing partition sum.

\bigskip

To study the limiting behaviour of the measures $\mu_n$ for large $n$, we first embed our spin configurations into a common space of measures. For any Polish space $E$, let $C_b(E)$ be the space of bounded continuous functions on $E$ and let $\cM_+(E)$ and $\cP(E)$ be the space of non-negative, respectively probability, Borel measures on $E$. We equip both spaces with the weak topology, i.e. the metric topology induced by testing against functions in $C_b(E)$. Recall that if $E$ is compact, then also $\cP(E)$ is weakly compact.

For $\Lambda \subseteq \Delta^d_n$ let $\pi_{\Lambda}: \Omega_{q,n} \mapsto \cP(\mathbb{T}^d_n\times\{1,\dots,q\})\subseteq \cP(\mathbb{T}^d\times\{1,\dots,q\})$ be the \textit{empirical colour measure vector} or \textit{colour profiles }of $\sigma$ inside the volume $\Lambda$ defined by
\begin{equation} \label{def:pi_Lambda}
\pi_\Lambda(\sigma):=\frac{1}{|\L|}\left(\sum_{x\in\L}1_{\s(x)=1}\d_{x/n},\dots,\sum_{x\in\L}1_{\s(x)=q}\d_{x/n}\right)
\end{equation} 
where $\delta_u$ is the Dirac measure at $u\in\mathbb{T}^d$.

For any $\nu\in \cP(\mathbb{T}^d\times\{1,\dots,q\})$ we will write $\nu[a]$ to indicate the evaluation of $\nu$ at a colour $a\in\{1,\dots,q\}$. In particular, for $\Lambda \subseteq \Delta_n^d$ and $\sigma \in \Omega_{q,n}$:
\begin{equation*}
\pi_\Lambda[a](\sigma):=\frac{1}{|\Lambda|}\sum_{x\in\Lambda}1_{\sigma(x)=a}\d_{x/n} \in \cM_+(\bT^d).
\end{equation*}
For simplicity, we will write $\pi_n$ and $\pi_n[a]$ to denote $\pi_{\Delta_n^d}$ and $\pi_{\Delta_n^d}[a]$. Using colour profiles, we can rewrite the Hamiltonian as
\begin{equation}
H_n(\sigma)=-n^d \sum_{a=1}^qF(\pi_{n}[a](\sigma))
\end{equation}
with $F(\nu[a]):=\langle J\ast \nu[a],\nu[a]\rangle=\int\int \nu[a](du)\nu[a](dv)J(u-v)$ and where the convolution of a function and a measure is defined as $(f \ast \mu)(x) = \int f(x-y) \mu(\dd y)$.

\smallskip

We will be interested in weak limits of color profiles in $\cP(\bT^d \times \{1,\dots,q\})$, especially those  having $q$-dimensional Lebesgue densities of the form $\nu=\alpha\lambda=(\alpha[1]\lambda,\dots,\alpha[q]\lambda)$ with $\alpha\in B$ where 
\begin{multline*}
B :=\left\{\alpha=(\alpha[1],\dots,\alpha[q])^T \middle| \text{ } 0\leq\alpha[a]\in L^\infty(\bT^d,\lambda) \phantom{\sum^{i}_j} \right. \\
\text{ with }
\left. \sum_{a=1}^q\alpha[a](x)=1\text{ for }\lambda\text{-a.e. }x\in\mathbb{T}^d \right\}.
\end{multline*}

In the sequel, we will write 
\begin{equation*}
\cP_q(B) := \left\{ \nu \in \cP(\bT^d \times \{1,\dots,q\}) \, \middle| \, \exists \alpha \in B: \, \nu = \alpha \lambda  \right\}.
\end{equation*}

We have the following straightforward result, proven for completeness in Appendix \ref{appendix:absolute_continuity}, which indicates that $\cP_q(B)$ is the relevant space of limiting profiles.
 
\begin{lemma} \label{lemma:possible_limits_of_profiles}
If $\sigma_n \in \Omega_{q,n}$ is a sequence such that $\pi_n(\sigma_n)$ converges weakly to $\nu$ in $\cP(\bT^d \times\{1,\dots,q\})$, then $\nu \in \cP_q(B)$.
\end{lemma}

\subsection{Sequential non-Gibbsianness for the Kac models}

For $n \geq 1$ and $u \in \bT^d_n$, we call $\pi_n^{(u)}:=\pi_{\Delta_n^d\setminus\{nu\}}$ the \textit{color profile perforated at} $u \in \bT^d_n$. We abbreviate $\mathcal{M}_{q,n}:=\pi_n(\Omega_{q,n})\subseteq \cP(\bT^d \times \{1,\dots,q\})$ and $\mathcal{M}_{q,n}^u:=\pi_n^{(u)}(\Omega_{q,n}) \subseteq \cP(\bT^d \times \{1,\dots,q\})$ for the sets of possible profiles of mesh-size $n^{-1}$ and possible profiles of mesh-size $n^{-1}$ perforated at site $u$. Note that by Lemma \ref{lemma:possible_limits_of_profiles} any limit $\nu$ of profiles $\nu_n \in \cM_{q,n}$ or $\cM_{q,n}^{(u_n)}$ must lie in $\cP_q(B)$.

\smallskip

We slightly change the definition of sequential Gibbsianness defined for Kac models as in \cite{JaKu17} and \cite{FHM14}. This definition is essentially equivalent but notationally and mathematically more convenient. 

\begin{definition}\label{GnG}
Given any sequence $(\mu_n)_{n\in\N}$ with $\mu_n$ a probability measure on $\Omega_{q,n}$ for every $n\in\bN$, define the single-spin conditional probabilities at site $u\in \bT_n^d$ as 
\begin{equation}\label{singesitespec}
\gamma^u_n\left(\cdot \, \middle| \, \nu^{(u)}_{n}\right):=\mu_n\left(\sigma(n u) =\cdot \, \middle| \, \pi^{(u)}_n(\sigma) = \nu^{(u)}_{n}\right) \hspace{1cm} \nu^{(u)}_{n}\in\mathcal{M}^u_{q,n}.
\end{equation}
\begin{enumerate}[(a)]
\item We call a colour profile $\nu \in \cP_q(B)$ \textbf{good} for a sequence of probability measures $(\mu_n )_{n\in\N}$, $\mu_n \in \cM_{q,n}$, if there exists a neighbourhood $\mathcal{N}_\nu \subseteq \cP(\bT^d \times \{1,\dots,q\})$ of $\nu$ such that for all $\tilde\nu \in\mathcal{N}_\nu \cap \cP_q(B)$ and for all $u\in\bT^d$
\begin{equation} \label{eqn:def_limiting_kernel}
\gamma^u(\cdot \, | \, \tilde\nu):=\lim_{n\uparrow\infty}\gamma^{u_n}_n(\cdot \, | \, \nu^{(u_n)}_{n})
\end{equation}
exists for all sequences $u_n \in \bT_n^d$ with $u_n \rightarrow u$ and all sequences $(\nu^{(u_n)}_{n})_{n\in\N}$ with $\nu^{(u_n)}_{n}\in\cM^{u_n}_{q,n}$ for every $n\in\N$ such that $\lim_{n\uparrow\infty}\nu^{(u_n)}_{n}=\tilde\nu$ in the weak sense. Moreover the limit must be independent of the choice of $u_n$ and  $(\nu^{(u_n)}_{n})_{n\in\N}$. 
\item A colour profile $\nu \in \cP_q(B)$ is called \textbf{bad} for $(\mu_n)_{n\in\N}$ if it is not good for $(\mu_n)_{n\in\N}$.
\item $(\mu_n)_{n\in\N}$ is called sequentially Gibbs if it has no bad profiles in $\cP_q(B)$. 
\end{enumerate}
\end{definition}

\begin{remark}
By Lemma \ref{lemma:continuity_of_kernel_abstract} below, it follows that if $(\mu_n)_{n\in\N}$ is sequentially Gibbs, then the map $\gamma : \bT^d \times \cP_q(B) \rightarrow \cP(\{1,\dots,q\})$, defined by $\gamma(u,\nu) = \gamma^{u}(\cdot \, | \, \nu)$ as in \eqref{eqn:def_limiting_kernel} is continuous.
\end{remark}

\subsection{Sequential Gibbsianness for the fuzzy Kac-Potts model} \label{section:seq_gibbs_for_FKP}

Next, we introduce the fuzzy Kac-Potts model. Consider a discretization map $T:\{1,\dots,q\}\mapsto\{1,\dots,s\}$ where $1<s<q$. More precisely, let $R_1,\dots,R_s$ be a partition of $\{1,\dots,q\}$ with $r_i=|R_i|$ and $\sum_{i=1}^sr_i=q$, then $T(a)=i$ if $a\in R_i$. The map $T$ induces a local discretization map $T : \Omega_{q,n} \rightarrow \Omega_{s,n}$ by applying $T$ at every site. The \textit{fuzzy Kac-Potts model} is obtained from the Kac-Potts model by applying the discretization map at every site: $\mu^T_n:=\mu_n \circ T^{-1} \in \cP(\Omega_{s,n})$, where $\mu^n$ is the Kac-Potts measure.

\begin{definition}
We call the generalized fuzzy KPM sequentially Gibbs if all profiles $\cP_s(B)$
are good for the sequence $\mu_n^T$.
\end{definition}

Jahnel and Külske studied the question of sequential Gibbsianness in \cite{JaKu17}. Their main result states that for the fuzzy KPM the critical parameters for GnG are the same as for the mean-field fuzzy PM if the parameters are such that low temperature Ising classes are avoided. Denote by $\beta_c(r)$ the inverse critical temperature of the $r$-state mean-field PM.

\begin{theorem}[Theorem 2.7 in \cite{JaKu17}]\label{HaKu-Kac}
Consider the $q$-state KPM at inverse temperature $\b$ and let $s$ and $r_1,\dots,r_s$ be positive integers with $1<s<q$ and $\sum^s_{i=1} r_i=q$. Consider the limiting conditional probabilities of the corresponding fuzzy KPM with spin partition $(r_1,\dots,r_s)$ where $r*:=\min\{r\geq3,r=r_i\text{ for some }i=1,\dots,s\}$.

\begin{enumerate}[(a)]
	\item Suppose that either $\b\leq \b_c(2)$ or that $r_i\not = 2$ for all $i=1,\dots,s$ and $\b<\b_c(r*)$, then the fuzzy KPM is sequentially Gibbs. 
	\item If $r_i\geq3$ for some $i=1,\dots,s$ and $\b\ge\b_c(r_*)$, then the fuzzy KPM is non-Gibbs. 
\end{enumerate}
\end{theorem}

This result should be compared to Theorem 1.2 in \cite{HK04}, which establishes a full characterization of the parameter regions in which the sequence of measures is sequentially Gibbs or not. Theorem \ref{HaKu-Kac} leaves open the question what happens if there are fuzzy classes of size $2$ and $\beta \geq \beta_c(2)$. The main result of this paper establishes that the fuzzy Kac-Potts model is also sequentially Gibbs if all classes are of size $2$ or $\beta<\beta_c(r*)$, establishing a full characterization with the same parameter regime as for the mean-field fuzzy Potts model.

\begin{theorem}\label{theorem:fuzz_kac_potts_gng}
	Consider the $q$-state Kac-Potts model at inverse temperature $\b$, 
	and let $s$ and $r_1,\dots,r_s$ be positive integers with $1<s<q$ and $\sum^s_{i=1} r_i=q$. Consider the limiting conditional probabilities of the corresponding fuzzy Kac-Potts model with spin partition $(r_1,\dots,r_s)$.
	\begin{enumerate}[(a)]
		\item Suppose that $r_i\leq2$ for all $i=1,\dots,s$. Then the fuzzy Kac-Potts model is sequentially Gibbs for all $\b\geq0$.
	\end{enumerate}
	\noindent  
	Assume that $r_i\geq3$ for some $i$ and put $r*:=\min\{r\geq3,r=r_i\text{ for some }i=1,\dots,s\}$, then the following holds. 
	\begin{enumerate}
		\item[(b)] If $\b<\b_c(r*)$ then the fuzzy Kac-Potts model is sequentially Gibbs.
		\item[(c)] If $\b\geq\b_c(r*)$ then the fuzzy Kac-Potts model is not sequentially Gibbs.
	\end{enumerate}
\end{theorem}

Like in the proof of Theorem 2.7 in \cite{JaKu17}, i.e. Theorem \ref{HaKu-Kac}, the analysis is based on a representation of the single-site kernels in terms of diluted Potts models, and an identification of the minimizers of an LDP for these diluted models. Unlike in the setting of classes of size larger or equal to three, for the `low temperature' Ising classes we find exactly two minimizers with opposite local magnetization that are not spatially homogeneous. To give a representation for the limiting kernels in Theorem \ref{theorem:form_of_kernels_in_Gibbs_case} below, we therefore need to introduce the corresponding representation and large deviation principle.

\subsection{A representation for the fuzzy Potts-Kac model kernels and a diluted large deviation principle} \label{section:representation_kernels_and_LDP}

In order to determine Gibbsianness of the fuzzy KPM, similar to \eqref{singesitespec}, as well as to give a representation for the limiting single-site kernels, we write for the single-site kernels
\begin{equation}\label{eqn:Representation_Kernel_FPK}
\gamma^{u}_{n,\beta,q,(r_1,\dots,r_s)}(k \,| \, \nu):=\mu_n^T(\sigma(nu)=k \, | \, \pi_{n}^{(u)}(\sigma)=\nu)
\end{equation}
where $\beta$ is the inverse temperature of the KPM and $\nu\in\cM^u_{s,n}$. It was shown in \cite{JaKu17} that the kernels can be re-expressed in terms of certain functionals integrated over KPM models on the fuzzy classes. 

\begin{definition} \label{definition:pre_kernel_A_expression}
	Definite the following three objects:
	\begin{enumerate}[(a)]
		\item For a measure $\nu \in \cM_{s,n}$, $u \in \bT_n^d$ and $i \in \{1,\dots,s\}$, denote by
		\begin{equation*}
		\Lambda_{i,u}(\nu) := \left\{x \in \Delta_n^d \setminus \{un\} \, \middle| \,  \nu[i](x/n) = 1 \right\}
		\end{equation*}
		the set of sites with a spin-value in the $i$-th class. 
		\item $\mu_{\Lambda,\beta,r}$ denotes the KPM in the subvolume $\Lambda\subseteq \Delta^d_n$, $\Lambda \neq \emptyset$, with Hamiltonian 
		\begin{equation*}
		H_{\Lambda}(\sigma):=-\frac{1}{|\Lambda|}\sum_{x,y\in \Lambda}J\left(\frac{x-y}{n}\right)1_{\sigma(x)=\sigma(y)},
		\end{equation*}
		inverse temperature $\beta$ and $r$ local states. 	Denote by $P_{\Lambda,\beta,r} \in \cP(\cP(\bT^d \times \{1,\dots,r\}))$ the push-forward of $\mu_{\Lambda,\beta,r}$ under $\pi_{\Lambda}$:
		\begin{equation*}
		P_{\Lambda,\beta,r} := \mu_{\Lambda,\beta,r} \circ \pi_{\Lambda}^{-1}.
		\end{equation*}
		\item For each class size $r$, define $\cA_r : \cP(\bT^d \times \{1,\dots,r\}) \times \bT^d \times \bR^+  \rightarrow [r,\infty)$as
		\begin{equation*}
		\cA_r(\pi,u,\beta) := \sum_{i=1}^r \exp \left\{2 \beta \left( J \ast \pi[i]\right)(u) \right\}.
		\end{equation*}
	\end{enumerate}
\end{definition}

Note that the functional $\cA$ that we introduce, is not equal to $A$ in \cite{JaKu17}. The functional $A$ includes an integral, which $\cA$ does not. This is done to allow for a better control on and representation of the limiting problem, and the limiting kernels to be considered below.

\begin{proposition}[Proposition 2.4 in \cite{JaKu17}]\label{proposition:FuzzyKernel}
	Fix $n$,$u\in\mathbb{T}^d_n$ and $\nu \in \cM_{s,n}$. Write $\beta_{l,u}(\nu) = \beta n^{-d} |\Lambda_{l,u}(\nu)| = \beta \nu[l](\bT_n^d \setminus \{u\})$ for the renormalized inverse temperature in class $l$. Suppose that $\Lambda_{l,u}(\nu) \neq \emptyset$ for all $l$. Then we have the representation 
	\begin{equation}\label{Representation_Kernel}
	\gamma^{u}_{n,\beta,q,(r_1,\dots,r_s)}(k \,| \,\nu) = \frac{\int \cA_{r_k}(\pi,u,\beta_{k,u}(\nu)) P_{\Lambda_{k,u}(\nu),\beta_{k,u}(\nu),r_k}(\dd \pi)}{\sum_{l=1}^s \int \cA_{r_l}(\pi,u,\beta_{l,u}(\nu)) P_{\Lambda_{l,u}(\nu),\beta_{l,u}(\nu),r_l}(\dd \pi)}.
	\end{equation}
	For a classes $l$ of size $0$, i.e. $\Lambda_{l,u}(\nu) = \emptyset$, the weight
	\begin{equation*}
	\int \cA_{r_l}(\pi,u,\beta_{l,u}(\nu)) P_{\Lambda_{l,u}(\nu),\beta_{l,u}(\nu),r_l}(\dd \pi),
	\end{equation*}
	should be replaced with $\cA_{r_l}(\pi,u,0) = r_k$, consistent with the fact that $\Lambda_{l,u}(\nu) = \emptyset$ implies that $\beta_{l,u}(\nu) = 0$.
\end{proposition}

\begin{remark} \label{remark:empty_class_1}
	In \cite{JaKu17}, classes of size $0$ were not treated. Their proof of the representation, however, shows that indeed the weight equals the size of the class.
\end{remark}

We find that, to study the limiting behaviour of the kernels, we need to study the limiting behaviour of the measures $P_{\Lambda_{l,u}(\nu_n),\beta_{l,u}(\nu_n),r_l}$ when $\nu_n \rightarrow \nu$ and, additionally, continuity properties of $\cA_r$ when $(\beta_n,u_n) \rightarrow (\beta,u)$. We start by considering $\cA_r$.

\begin{lemma} \label{lemma:continuity_of_cA}
	Let $r \geq 1$ and let $\cA_r$ be the map defined in Definition \ref{definition:pre_kernel_A_expression}. Then 
	\begin{enumerate}[(a)]
		\item For each pair $(u,\beta) \in \bT^d \times \bR^+$ the map $\pi \mapsto \cA_r(\pi,u,\beta)$ is an element of $C_b(\cP(\bT^d \times \{1,\dots,r\}))$,
		\item Let $\{(u_n,\beta_n)\}_{n \geq 1}$ be a sequence of pairs in $\bT^d \times \bR^+$ converging to $(u,\beta) \in \bT^d \times \bR^+$, then $\cA_r(\cdot,u_n,\beta_n)$ converges uniformly as a function on $\cP(\bT^d \times \{1,\dots,r\})$ to the function $\cA_r(\cdot,u,\beta)$.
	\end{enumerate}
\end{lemma}

We immediately obtain the following result, corresponding to the fact that $\int f_n \dd \mu_n \rightarrow \int f \dd \mu$ if $f_n$ converges uniformly to $f$ and $\mu_n$ weakly to $\mu$.

\begin{lemma} \label{lemma:convergence_of_kernels}
	Let $\nu_n \rightarrow \nu$ and $u_n \rightarrow u$. Suppose that for each fuzzy class $l \in \{1,\dots,s\}$ with $\nu[l](\bT^d) > 0$ there exists some measure $P_l$ such that $P_{\Lambda_{l,u_n}(\nu_n),\beta_{l,u_n}(\nu_n),r_l} \rightarrow P_l$ weakly. Then
	\begin{equation*}
	\gamma_{n,\beta,q,(r_1,\dots,r_s)}^{u_n}(k \, | \, \nu_n) \rightarrow \frac{\int \cA_{r_k}(\pi,u,\beta_k(\nu)) P^k(\dd \pi)}{\sum_{l=1}^s \int \cA_{r_l}(\pi,u,\beta_l(\nu)) P^l(\dd \pi)}.
	\end{equation*}
	As above, if a fuzzy class $l$ has no limiting mass, i.e. $\nu[l](\bT^d) = 0$, then the weight should be replaced by $\cA_{r_l}(\cdot,u,0) = r_l$.
\end{lemma}

To study the limiting properties of the sequences $P_{\Lambda_{l,u_n}(\nu_n),\beta_{l,u_n}(\nu_n),r_l}$, we proceed with the large deviation principle proven in \cite{JaKu17} for the measures $P_{\Lambda_{l,u_n}(\nu_n),\beta_{l,u_n}(\nu_n),r_l}$.

\begin{definition}
We say that a sequence of volumes $\Lambda_n \subseteq \Delta_n^d = n \bT^d_n$ converges to some measure $\rho \in \cM_+(\bT^d)$, denoted by $\Lambda_n \Rightarrow \rho$ if the empirical measure
\begin{equation*}
\frac{1}{n^d} \sum_{i \in \Lambda_n} \delta_{x/n} \in \cM_+(\bT^d)
\end{equation*}
converges weakly to $\rho$. Abusing notation, if $\rho$ has a density with respect to the Lebesgue measure, then we will denote this density by $\rho$.
\end{definition}

\begin{proposition}[Proposition 2.5 in \cite{JaKu17}]\label{proposition:DiLDP} (Diluted version of LDP for empirical color profiles).   
	Consider a sequence of inverse temperatures $\tilde{\beta}_n \rightarrow \tilde{\beta}$ and a sequence of diluted sets $\Lambda_n \subseteq \Delta_n^d$ with $\Lambda_n\Rightarrow\rho$ for some Lebesgue density $\rho$ with $N_\rho:=\rho\lambda(\mathbb{T}^d)>0$. Denote $\tilde\rho(u):=N_\rho^{-1}\rho(u)$, then the measures $\mu_{\Lambda_n,\tilde{\beta}_n,q}\circ(\pi_{\Lambda_n})^{-1}$ satisfy an LDP on the space $\cP(\bT^d \times \{1,\dots,q\})$ with rate $|\Lambda_n|$ and rate function $I_{\tilde\rho,\tilde{\beta}}-\inf_{\phi \in \cP(\bT^d \times \{1,\dots,q\})}  I_{\tilde\rho,\tilde{\beta}}(\phi)$ where 
	\begin{equation}\label{RateF}
	I_{\tilde\rho,\tilde{\beta}}(\phi)=\begin{cases}
	- \tilde{\beta} \sum_{a=1}^q \ip{J\ast\tilde\rho\alpha[a]}{\tilde\rho\alpha[a]} + \ip{S(\alpha|\mathrm{eq})}{\tilde\rho \lambda}  & \text{if }\phi[a]=\tilde\rho\alpha[a]\lambda,\alpha\in B\\
	\infty & \text{otherwise. }
	\end{cases}
	\end{equation}
\end{proposition}

\begin{remark}
	Proposition 2.5 in \cite{JaKu17} was originally proven for fixed $\tilde{\beta}$. Changing the result to include $\tilde{\beta}_n \rightarrow \tilde{\beta}$ follows immediately from Varadhan's lemma for a uniformly converging sequence of bounded continuous functions.
\end{remark}

\subsection{Minimizers of the rate function for Ising classes and a limiting form of the single-site kernels} \label{section:Intro_minimizers_and_limiting_kernels}

A more careful analysis of the rate function for the Ising classes gives an extension of the representation of the limiting kernels of Theorem 2.7 in \cite{JaKu17}. For any $\phi = \alpha \tilde{\rho}\lambda$, we first re-express $\alpha$ in terms of the local magnetization
\begin{equation*}
m^\alpha(u) := \alpha[1](u) - \alpha[2](u).
\end{equation*}

\begin{proposition} \label{proposition:minimizers_I_and_convergence}
	Consider the rate function in Proposition \ref{proposition:DiLDP} for an Ising class with $\Lambda_n \Rightarrow \rho \neq 0$. Denote $P_n := \mu_{\Lambda_n,\tilde{\beta}_n,2}\circ(\pi_{\Lambda_n})^{-1}$. Then the following two situations can occur.
	\begin{enumerate}[(a)]
		\item There is exactly one global minimizer $\phi^* = \phi^*(\tilde{\rho},\tilde{\beta})$ for $I_{\tilde{\rho},\tilde{\beta}}$, which is the spatially homogeneous equi-distribution, i.e. corresponding to local magnetization profile $m_{\tilde{\rho},\tilde{\beta}} = 0$. We have $P_n \rightarrow \delta_{\phi^*}$.
		\item There are exactly two global minimizers $\phi^{*,+} = \phi^{*,+}(\tilde\rho,\tilde{\beta})$ and $\phi^{*,-} = \phi^{*,-}(\tilde\rho,\tilde{\beta})$ for $I_{\tilde{\rho},\tilde{\beta}}$. Let $m_{\tilde{\rho},\tilde{\beta}}$ denote the local magnetization of $\phi^{*,+}$. Then $m_{\tilde{\rho},\tilde{\beta}}$ is everywhere positive, and the local magnetization of $\phi^{*,-}$ equals $- m_{\tilde{\rho},\tilde{\beta}}$. We have $P_n \rightarrow \frac{1}{2}(\delta_{\phi^{*,+}} + \delta_{\phi^{*,-}})$.
	\end{enumerate} 
\end{proposition}

Thus, for classes of size $2$, and $\rho \neq 0$, we have a well defined local magnetization $m_{\tilde{\rho},\tilde{\beta}}$. For a class of size $2$ and $\rho =0$, set $m_{\tilde{\rho},\tilde{\beta}} = 0$. Define for each $r \geq 2$ the function
\begin{equation*}
D_{r} : \bT^d \times P_r(B) \times \bR^+ \rightarrow [1,\infty), \quad 	D_{r}(u,\tilde{\rho},\tilde{\beta}) := \bONE_{\{r \neq 2\}}+\bONE_{\{{r=2}\}} \frac{1}{\sqrt{1-m_{\tilde{\rho},\tilde{\beta}}(u)^2}}.
\end{equation*}

\begin{theorem} \label{theorem:form_of_kernels_in_Gibbs_case}
	Consider the setting of Theorem \ref{theorem:fuzz_kac_potts_gng}. Let $\nu_n \rightarrow \nu$ and $u_n \rightarrow u$.  Suppose the parameters are in setting (a) or (b). Then the limiting conditioning kernel is given by
	\begin{equation*}
	\lim_{n \rightarrow \infty}{\gamma_{n,\beta,q,(r_1,\ldots,r_s)}^u(k\vert \nu_n)}= \frac{r_k \exp \left\{ 2\beta r_k^{-1}\int \dd v \, \nu[k](v)J(u-v)\right\} D_{r_k}(u,\tilde{\nu}_k,\tilde{\beta}_k)}{\sum_{i=1}^s{r_i \exp\left\{2\beta r_i^{-1}\int \dd v \, \nu[i](v)J(u-v)\right\} D_{r_i}(u,\tilde{\nu}_i,\tilde{\beta}_i)}}
	\end{equation*}
	where $\tilde{\nu}_i := \nu[i](\bT^d)^{-1} \nu[i]$ and $\tilde{\beta}_i = \nu[i](\bT^d) \beta$.
\end{theorem}

\section{Proofs of main results} \label{section:proofs}
	
	We start by proving our main Theorem \ref{theorem:fuzz_kac_potts_gng} using the results from Lemma's \ref{lemma:continuity_of_cA} and \ref{lemma:convergence_of_kernels} and Proposition \ref{proposition:minimizers_I_and_convergence}. We prove the two lemma's immediately afterwards.
	
	Then, in Sections \ref{section:intro_Ising_classes} to \ref{section:Ising_form_of_kernel}, we analyse of the minimizers of the rate function for Ising classes which leads to a proof of Proposition \ref{proposition:minimizers_I_and_convergence}. We conclude in Section \ref{section:Ising_convergence_of_profiles} with a proof of Theorem \ref{theorem:form_of_kernels_in_Gibbs_case}..

	\subsection{Proof of Theorem \ref{theorem:fuzz_kac_potts_gng}} \label{section:proof_main_theorem}
	
	We proceed with the proof of our main result, which is based on Lemma's \ref{lemma:continuity_of_cA} and \ref{lemma:convergence_of_kernels} and Proposition \ref{proposition:minimizers_I_and_convergence}.

	\begin{proof}[Proof of Theorem \ref{theorem:fuzz_kac_potts_gng}]
	First we prove (a), i.e. all classes have size $1$ or $2$. Let $\nu_n \rightarrow \nu$ and $u_n \rightarrow u$. For classes $l \in \{1,\dots,s\}$ and $n \geq 1$ such that $\Lambda_{l,u_n}(\nu_n) \neq \emptyset$, set $P^l_n := P_{\Lambda_{l,u_n}(\nu_n),\beta_{l,u_n}(\nu_n),r_l}$. 
		
		\smallskip
		
		To prove our result, it suffices to verify the conditions for Lemma \ref{lemma:convergence_of_kernels}. Note that we do not need to consider classes with no limiting mass. Thus, without loss of generality, we assume all fuzzy classes have non-zero limiting mass.
		
		First suppose that class $l$ has size $1$. Then it follows that the measure $P_n^l$ converge to $\tilde{\nu}_l$.  For classes of size $2$ with non-zero mass in the limit, we find by Proposition \ref{proposition:DiLDP} with $\Lambda_n = \Lambda_{l,u_n}(\nu_n)$, $\tilde\beta_n = \beta_{l,u_n}(\nu_n)$ and $\rho = \nu[l]$ that $P_n^l$ satisfies a large deviation principle with good rate-function. By Proposition \ref{proposition:minimizers_I_and_convergence} there is a measure $P^l$, independent of the sequences $\nu_n$ and $u_n$, such that $P_n^l \rightarrow P^l$. Thus Lemma \ref{lemma:convergence_of_kernels} implies the limit of the kernels exists and is independent of the sequences $\nu_n$ and $u_n$. 
%
%
%

	Cases (b) and (c) follow from the arguments of  \cite{JaKu17}, combined with the result of the present paper that even when Ising-classes are present, they do not provide a source of discontinuity. 

	\end{proof}

	\subsection{Proof of Lemma's \ref{lemma:continuity_of_cA} and \ref{lemma:convergence_of_kernels}} \label{section:proofs_main_lemmas}
	
	\begin{proof}[Proof of Lemma \ref{lemma:continuity_of_cA}]
		Fix $r \geq 1$ Recall that $\cA_r : \cP(\bT^d \times \{1,\dots,r\}) \times \bT^d \times \bR^+  \rightarrow [1,\infty)$ was defined as
		\begin{equation*}
		\cA_r(\pi,u,\beta) := \sum_{i =1}^r\exp \left\{2 \beta \left( J \ast \pi[i]\right)(u) \right\}.
		\end{equation*}
		We start with proving (a), the continuity of $\pi \mapsto \cA_r(\pi,u,\beta)$ for all $(u,\beta) \in \bT^d \times \bR^+$. Clearly, this result follows if $\pi \mapsto (J \ast \pi[i])(u)$ is continuous for all $i$. But this is immediate as $J$ is a continuous function and $\cM_+(\bT^d)$ is equipped with the weak topology.
		
		We proceed with the proof of (b). Let $(u_n,\beta_n) \rightarrow (u,\beta)$. Note that the uniform convergence, again, follows by proving that the function $f_{n,i}(\pi) := (J \ast \pi[i])(u_n)$ converges uniformly to $f_i(\pi) := (J \ast \pi[i])(u)$. This, however, follows immediately from the uniform continuity of the function $J$($\bT^d$ is compact).
	\end{proof}

\begin{proof}[Proof of Lemma \ref{lemma:convergence_of_kernels}]
	Let $\nu_n \rightarrow \nu$ and $u_n \rightarrow u$. By the representation for the kernels given in Proposition \ref{proposition:FuzzyKernel}, using that the weight of class $l$ is bounded below by $r_l$, we find that the limiting statement holds if for all classes $l \in \{1,\dots,s\}$, we have that
	\begin{equation} \label{eqn:convergence_of_kernel_integrals}
	\int \cA_{r_l}(\pi,u_n,\beta_{l,u_n}(\nu_n)) P^l_n(\dd \pi) \rightarrow \int \cA_{r_l}(\pi,u,\beta_l) P^l(\dd \pi),
	\end{equation}
	where the weights need to be replaced by $r_l$ if the set $\Lambda_{l,u_n}(\nu_n) = \emptyset$ on the left-hand side, or if $\nu[l](\bT^d)$ on the right-hand side.
	
	\smallskip
	
	Fix $l$. First, suppose that $\nu[l](\bT^d) = 0$ and hence $\Lambda_{l,u_n}(\nu_n) \Rightarrow 0$. If the set $\Lambda_{l,u_n}(\nu_n) = \emptyset$, then the weight in \eqref{Representation_Kernel} equals $r_l$. If the set is not empty, then, we need to consider the integral 
	\begin{equation*}
	\int \cA_{r_l}(\pi,u,\beta_{l,u_n}(\nu_n)) P_{\Lambda_{l,u_n}(\nu_n),\beta_{l,u_n}(\nu_n),r_l}(\dd \pi). 
	\end{equation*}
	Regardless of which case we have, we have that $\beta_{l,u_n}(\nu_n) \rightarrow 0$, and hence $\cA_{r_l}(\cdot,\beta_{l,u_n}(\nu_n),u_n) \rightarrow \cA_{r_l}(\cdot,0,u) = r_l$ uniformly by Lemma \ref{lemma:continuity_of_cA}. This implies that the weight of class $l$ has a unique limit, which equals $r_l$.
	
	Next, we assume that $\nu[l](\bT^d) > 0$. This implies that for sufficiently large $n$, the set $\Lambda_{l,u_n}(\nu_n) \neq \emptyset$. As $\beta_{l,u_n}(\nu_n) = \beta n^{-d} |\Lambda_{l,u_n}(\nu_n)| = \beta \nu_n[l](\bT^d)$, i.e. we find by the weak convergence of $\nu_n \rightarrow \nu$ that $\beta_{l,u_n}(\nu_n) \rightarrow \beta_l := \beta \nu[l](\bT^d)$. Because also $u_n \rightarrow u$, we find that the continuous functions $\cA_{r_l}(\cdot, u_n,\beta(\nu_n))$ converge uniformly to the continuous function $\cA_{r_l}(\cdot,u,\beta)$ by Lemma \ref{lemma:continuity_of_cA}. Thus, the weak convergence $P^l_n \rightarrow P^l$ establishes \eqref{eqn:convergence_of_kernel_integrals}.
\end{proof}

\subsection{Preliminaries for the Ising-class analysis} \label{section:intro_Ising_classes}

Before identifying minimizers of the rate-function of Proposition \ref{proposition:DiLDP} for Ising classes, we first reparametrize our Ising profiles in terms of the magnetization of the profile. Afterwards, we rewrite our rate function in terms of a local term, analogous to that of the Curie-Weiss model, and a global term that expresses the non-local interactions.

\smallskip

Consider $\tilde{\beta} \geq 0$ and a profile $\rho\lambda \in \cM_+(\bT^d)$ with Lebesgue density $\rho$ and assume that $N_\rho := \rho \lambda(\bT^d) > 0$. Set $\tilde{\rho} = N_\rho^{-1} \rho$ as in Proposition \ref{proposition:DiLDP}. We study the minimizers of $I_{\tilde{\rho},\tilde{\beta}}$, which is equivalent to studying the mimimizers of the rate-function $I_{\tilde{\rho},\tilde{\beta}} - \inf_{\phi} I_{\tilde{\rho},\tilde{\beta}}(\phi)$, where
\begin{equation*}
\begin{split}
I_{\tilde{\rho},\tilde{\beta}}(\phi)
=\begin{cases} 
&\int \dd u \tilde{\rho}(u) \left\{ -\tilde{\beta} \int \dd v \tilde{\rho}(v)\left(\alpha[1](u)\alpha[1](v)+\alpha[2](u)\alpha[2](v)\right)J(u-v) \right. \\
& \hspace{3cm} + \left. S(\alpha[\cdot](u)\ \vert eq)\right\} \quad \text{if } \phi =\alpha \tilde{\rho} \lambda, \alpha \in B_2 \\ 
& \infty \hspace{5.9cm} \text{else} 
\end{cases}
\end{split}
\end{equation*} 
Thus, $I_{\tilde{\rho}}(\phi)$ is expressed in terms of $\alpha \in B_2$, where $\phi = \alpha \tilde{\rho} \lambda$. First, we will rewrite the rate-function in terms of the local magnetization $m$ in the closed ball 
\begin{equation*}
\text{Ball}(L^\infty) := \left\{m \in L^\infty(\mathbb{T}^d, \lambda^d) \, \middle| \, \esssup \vert m \vert \leq 1 \right\},
\end{equation*}
defined by $m(u) := \alpha[1](u) - \alpha[2](u)$. To re-express the diluted rate-function in terms of magnetization functions, we express the quadratic term in terms of $m$:
\begin{equation*}
\alpha[1](v)\alpha[1](u)+\alpha[2](u)\alpha[2](v)=\frac{1+ m(u)m(v)}{2}.
\end{equation*}
Substituting the representation into the rate-function gives
\begin{equation*}
I_{\tilde{\rho},\tilde{\beta}}(\phi)=\int \dd u \tilde{\rho}(u)\left\{-\frac{\tilde{\beta}}{2} \int \dd v \tilde{\rho}(v)(1+ m(u)m(v))J(u-v)+I(m(u))\right\}
\end{equation*} 
where $I$ denotes the entropy-term in the Curie-Weiss rate-function
\begin{equation*}
I(x)= \begin{cases}
\frac{1+x}{2}\log(1+x)+\frac{1-x}{2}\log(1-x) &  \text{if } x \in (-1,1), \\
\log(2)  & \text{if } x \in \{-1,+1\}.
\end{cases} 
\end{equation*}
Note, that, since the final rate-function equals $I_{\tilde{\rho},\tilde{\beta}}$ plus some constant we can omit any terms in the integrand which do not depend on $m$. We will further write $I_{\tilde{\rho},\tilde{\beta}}(m)$ instead of $I_{\tilde{\rho},\tilde{\beta}}(\phi)$ if $\phi$ has a magnetization represented by $m$. We conclude that it suffices to study minimizers of the functional
\begin{equation} \label{newRatef}
I_{\tilde{\rho},\tilde{\beta}}(m) = \int \dd u \tilde{\rho}(u)\left\{-\frac{\tilde{\beta}}{2} \int \dd v \tilde{\rho}(v)m(u)m(v)J(u-v)+I(m(u))\right\}
\end{equation}
on the set $\text{Ball}(L^\infty)$. Analogous to the discussion in the case of all classes of at size of at  least three, we may rewrite the rate-function \eqref{newRatef} by
\begin{multline} \label{splitRF}
I_{\tilde\rho,\tilde{\beta}}(m)= \frac{\beta}{4} \int \dd u \tilde\rho(u)\int \dd v \tilde\rho(v)\left[m(u)-m(v)\right]^2 J(u-v) \\
+ \int \dd u\tilde\rho(u)\left[-\frac{1}{2} b_{\tilde{\beta},\tilde\rho,J}(u)m^2(u)+I(m(u))\right] 
\end{multline}
with the local inverse temperature at  site $u$ given by the convolution
\begin{equation}
b_{\tilde{\beta},\tilde\rho,J}(u):=\tilde{\beta} \int \dd v \tilde\rho(v)J(u-v).
\end{equation}
To elucidate the equivalence between the bracketed expression in the second term of $I_{\tilde{\rho}}$ and the Curie-Weiss rate-function at (site-dependent) inverse temperature $b_{\beta,\tilde{\rho},J}(u)$ and magnetization $m(u)$, we further use the notation 
\begin{equation*}
\Phi_u(m)=:-\frac{1}{2}b_{\tilde{\beta},\tilde{\rho},J}(u)m^2+I(m), \quad  u \in \mathbb{T}^d, m \in [-1,1].
\end{equation*}
We see that the rate functional expresses the competition between the local Curie-Weiss term and a global term that penalizes spatial inhomogeneity.

\smallskip

To visualize this competition, Figure \ref{figure:profiles} shows three different magnetization profiles on the one-dimensional torus at fixed interaction-function $J$, (conditioning) density $\tilde{\rho}$ and inverse temperature $\tilde{\beta}$. The profile $m_\text{loc}$ (see Lemma \eqref{lemma:minimizers_are_nonnegative}) is given as 
\begin{equation*}
m_\text{loc}(u):=\argmin_{m \in [0,1]}(\Phi_u(m)),
\end{equation*}
whereas $m_\text{flat}$ is the minimizer to the local Curie-Weiss term in the class of non-negative spatial homogeneous profiles.
The profile $m_\text{stat}$ arises from the necessary condition  of vanishing gradient at minimizing profiles to $I_{\tilde{\beta},\tilde{\rho}}$.
\begin{center}
\begin{figure}[h]
 \centering
 \includegraphics[width=0.9\textwidth]{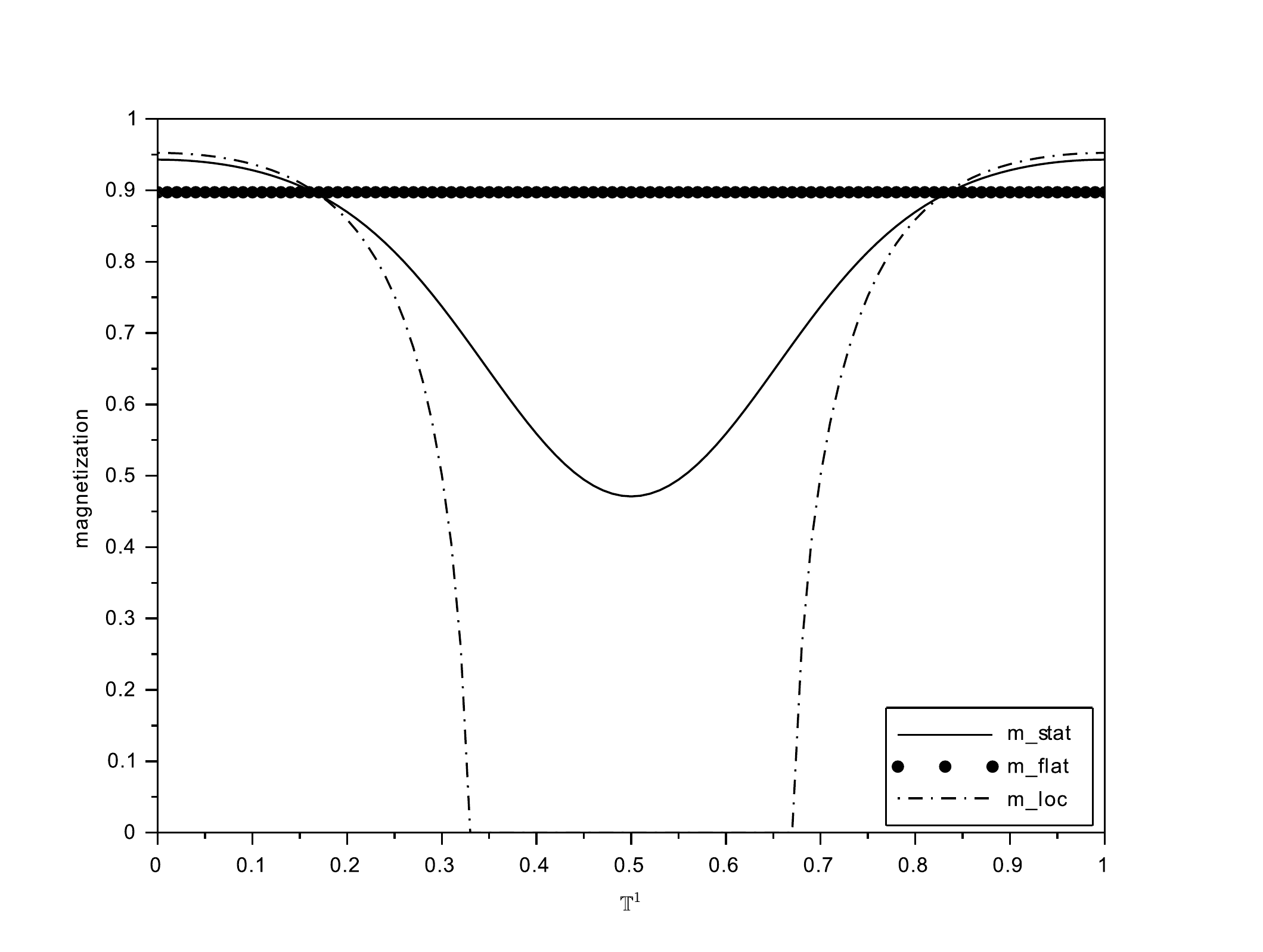}
 \caption{$\tilde{\beta} = 1.3$, $\tilde{\rho}(\cdot) = 1 +\cos(2\pi \cdot)$, $J(\cdot) = 1 + \cos(2\pi \cdot)$}
 \label{figure:profiles}
\end{figure} 
\end{center}


\subsection{Identifying the minimizers of the rate function for Ising classes}

We exploit the representation of our rate function in terms of the local magnetization $m$ to study minimizers of $I$. We first argue based on the assumption that a local minimizer exists. Based on this assumption and natural symmetry properties in the rate-function, we show that we can find a local minimizer with lower cost that is non-negative. In addition, we show, using the decomposition of $I$ into a local and global term that local minimizers are bounded away from $-1$ and $1$.

\begin{lemma} \label{lemma:minimizers_are_nonnegative}
	A local (global) minimizer $m$ of the rate function $I_{\tilde{\rho},\tilde{\beta}}$ has the following properties:
	\begin{enumerate}[(a)]
		\item By symmetry of $I_{\tilde{\rho},\tilde{\beta}}$, the profile $-m$ is also a local (global) minimizer.
		\item We can always find a non-negative profile $\tilde{m}$, such that $I_{\tilde{\rho},\tilde{\beta}}(\tilde{m}) \leq I_{\tilde{\rho}}(m)$. if both $\tilde{\rho}\lambda(\{m>0 \})>0$ and $\tilde{\rho}\lambda(\{m<0\})>0$ then the  inequality is strict.
		\item $m$ is an interior point of $\text{Ball}(L^\infty) $, i.e. $\esssup(\vert m \vert )<1$.
	\end{enumerate}
\end{lemma}

\begin{proof}
	(a) is immediate. For (b) consider the profile $\tilde{m} :=\vert m \vert$ given as the point-wise absolute value of $m$. The entropy functional $I(\cdot)$ is symmetric, implying
	\begin{equation*}
	\int \dd u \tilde{\rho}(u) I(\tilde{m}(u))= \int \dd u \tilde{\rho}(u)I(m(u)).
	\end{equation*}
	At any $(u,v) \in \mathbb{T}^d\times \mathbb{T}^d$ we further have 
	\begin{equation*}
	-\tilde{m}(u)\tilde{m}(v)
	\begin{cases}
	= -m(u)m(v) & \text{if } (u,v) \in \{m \geq 0\}^2 \cup \{m < 0 \}^2, \\
	< -m(u)m(v) & \text{else},
	\end{cases}
	\end{equation*}
	which leads to
	\begin{equation*}
	\begin{split}
	&\int du  \tilde{\r}(u) \left\{-\frac{\tilde{\beta}}{2}\int \dd v \tilde{\rho}(v)\tilde{m}(u)\tilde{m}(v)J(u-v) \right\} \\
	&\leq \int \dd u  \tilde{\rho}(u) \left\{-\frac{\tilde{\beta}}{2}\int dv \tilde{\rho}(v)m(u)m(v)J(u-v) \right\}
	\end{split}
	\end{equation*}
	with the inequality being strict if $\tilde{\rho}\lambda(\{m > 0\})>0$ and $\tilde{\rho}\lambda(\{m <0 \})>0$.	We conclude
	\begin{equation*}
	I_{\tilde{\rho},\tilde{\beta}}(\tilde{m}) \leq I_{\tilde{\rho},\tilde{\beta}}(m)
	\end{equation*}
	with strict inequality in the case discussed above. 
	
	\smallskip
	
	For the proof of (c) let $m\geq 0$ be any magnetization-profile with $\esssup(m)=1$. Consider the representation \eqref{splitRF} of $I_{\tilde{\rho},\tilde{\beta}}$ and define 
	\begin{equation*}
	m_\text{loc}(u):=\argmin_{m \in [0,1]}(\Phi_u(m)).
	\end{equation*}
	By continuity, the local inverse temperature $b_{\tilde{\beta},J,\tilde{\rho}}$ is bounded on $\mathbb{T}^d$, so $m_\text{loc}$, given as non-negative solution to the mean-field equation 
	\begin{equation} \label{MFeq}
	m_\text{loc}(u)=\tanh(b_{\tilde{\beta},J,\tilde{\rho}}(u)m_\text{loc}(u)),
	\end{equation}
	is bounded from above by a constant $c < 1$. Since $\esssup(m)=1$, we clearly have  $\tilde{\rho}\lambda(m>c)>0$. Thus, the profile $\tilde{m}(\cdot)=m(\cdot) \land c$ is an interior point of $\text{Ball}(L^\infty) $. By construction, the contribution of the local part of \eqref{splitRF} for $\tilde{m}$ is lower than for $m$. By a straightforward verification, the same follows for the global part, as $|\tilde{m}(u) - \tilde{m}(v)| \leq |m(u) - m(v)|$ for all $u,v$. We find $I_{\tilde{\rho},\tilde{\beta}}(\tilde{m})<I_{\tilde{\rho},\tilde{\beta}}(m)$.
\end{proof}

By (c) of Lemma \ref{lemma:minimizers_are_nonnegative}, we know that mimimizers lie in the interior of $\text{Ball}(L^\infty)$. This implies we can derive the function $I_{\tilde{\rho},\tilde{\beta}}$ to get further conditions on minimizers. We start with a technical lemma that identifies the gradient of $I_{\tilde{\rho}}$.

\begin{lemma} \label{lemma:derivative}
	For any two magnetization-profiles $m_1$ and $m_2$ and any $t \in [0,1]$ such that there is an open neighborhood $t \in U \subseteq \mathbb{R}$ with
	\begin{equation*}
		m_1+\hat{t}m_2 \in \text{Ball}(L^\infty)  \quad \text{for all } \hat{t} \in U 
	\end{equation*}
	the derivative $\frac{d}{d \tilde{t}}I_{\tilde{\rho},\tilde{\beta}}(m_1+\tilde{t}m_2)\vert_{\tilde{t}=t}$ exists and is given by
	\begin{multline*}
	\frac{\dd}{\dd \tilde{t}} I_{\tilde{\rho},\tilde{\beta}} (m_1+\tilde{t}m_2) \vert_{\tilde{t}=t} \\
	= \int  \dd  u \tilde{\rho}(u)m_2(u) \left\{ - \tilde{\beta} \int \dd v \tilde{\rho}(v)(m_1+tm_2)(v)J(u-v)+I'((m_1+tm_2)(u)) \right\}.
	\end{multline*}
\end{lemma}

\begin{proof}
	We have
	\begin{multline*}
	\frac{\dd}{\dd\tilde{t}}I_{\tilde\rho}(m_1+\tilde{t} m_2)\vert_{\tilde{t}=t} \\
	= - \frac{\tilde{\beta}}{2} \frac{\dd}{\dd\tilde{t}} \int \dd u \tilde{\rho}(u) \int \dd v \tilde{\rho}(v)(m_1+\tilde{t}m_2)(u)(m_1+\tilde{t}m_2)(v) J(u-v)\vert_{\tilde{t}=t} \\
	+ \frac{\dd}{\dd \tilde{t}}\int \dd u \tilde{\rho}(u)I((m_1+\tilde{t}m_2)(u)) \vert_{\tilde{t}=t}.
	\end{multline*}
We first consider the first term on the right-hand side. As the integrand in the first term is continuously differentiable in $\tilde{t}$ and that both $m_1$ and $m_2$ are essentially bounded, we can interchange integration and differentiation, cf. Lemma 6.28 in \cite[Lemma 6.28]{Kl14} or \cite[Exercise 5.8.135]{Bo07}:
	\begin{multline*}
	- \frac{\tilde{\beta}}{2}\int \dd u \tilde{\rho}(u) \int \dd v \tilde{\rho}(v)\frac{\dd}{\dd\tilde{t}} \left\{(m_1+\tilde{t}m_2)(u)(m_1+\tilde{t}m_2)(v))J(u-v)\right\} \vert_{\tilde{t}=t} \\
	 = - \frac{\tilde{\beta}}{2}\int \dd u \tilde{\rho}(u)m_2(u) \int \dd v \tilde{\rho}(v)(m_1+tm_2)(v)J(u-v) \\
	 -  \frac{\tilde{\beta}}{2}\int \dd u \tilde{\rho}(u) (m_1+tm_2)(u)\int \dd v m_2(v)J(u-v).
	\end{multline*}
	Using the Fubini-Tonelli-theorem on the second term on the right-hand side and then switching roles of $u$ and $v$ we find
	\begin{multline*}
	- \frac{\tilde{\beta}}{2}\int \dd u \tilde{\rho}(u) m_2(u) \left\{\int \dd v \tilde{\rho}(u)	(m_1+tm_2)(v)(J(u-v)+J(v-u)) \right\} \\
	= - \tilde{\beta} \int \dd u \tilde{\rho}(u) m_2(u) \left\{\int \dd v \tilde{\rho}(u) (m_1+tm_2)(v)J(u-v)\right\}.
	\end{multline*}
	In the same way we have:
	\begin{multline*}
	\frac{\dd}{\dd \tilde{t}}\int \dd u \tilde\rho(u) I((m_1+\tilde{t}m_2)(u))\vert_{\tilde{t}=t} = \int \dd u\tilde\rho(u) \frac{\dd}{\dd \tilde{t}}\{I((m_1+\tilde{t}m_2)(u))\}\vert_{\tilde{t}=t} \\
	= \int \dd u \tilde{\rho}(u)m_2(u)I'(m_1(u)+tm_2(u)),
	\end{multline*}
	concluding the proof of the lemma.
\end{proof}

By differentiation of $I_{\tilde{\rho},\tilde{\beta}}$ at a local minimizer, we immediately find the following necessary condition for minimizers.

\begin{lemma} 
	A minimizer $m$ to $I_{\tilde{\rho},\tilde{\beta}}$ satisfies the stationarity equation
	\begin{equation}\label{SCzzz}
	-\tilde{\beta} \int \dd v\tilde\rho(v)	m(v)J(u-v)+ I'(m(u)) = 0
	\qquad \text{for } \tilde{\rho}\lambda \text{-a.e. } u \in  \mathbb{T}^d,
	\end{equation}
	or equivalently,
	\begin{equation} \label{newStateq}
	m(u) = \tanh\left(\tilde{\beta} \int \dd v \tilde{\rho}(v)m(v)J(u-v)\right) \qquad \text{for } \tilde{\rho}\lambda \text{-a.e. } u \in  \mathbb{T}^d.
	\end{equation}
\end{lemma}

\begin{proof}
	Let $m$ be a local minimum and $\tilde{m}$ an arbitrary profile.
	By Lemma \ref{lemma:minimizers_are_nonnegative} we may assume $m$ to be an interior point of ${lemma:derivative}$, so, by Lemma \ref{lemma:derivative}the derivative
	\begin{equation*}
	\frac{\dd}{\dd\tilde{t}} I_{\tilde\rho,\tilde{\beta}}(m+\tilde{t} \tilde{m})\vert_{\tilde{t}=0} 
	\end{equation*}
	exists and must vanish for $m$ being a local minimum of $I_{\tilde{\rho},\tilde{\beta}}$. Rewriting $I'=\arctanh$ leads to the second expression.
\end{proof}

An immediate consequence of the stationarity equation is a further restriction of the set of possible minimizers.

\begin{lemma} \label{lemma:strictly_positive_minimizers}
	A nontrivial non-negative minimizer to $I_{\tilde{\rho},\tilde{\beta}}$ is strictly positive.
\end{lemma}

\begin{proof}
	By assumption, $J > 0$, which implies by the stationarity equation \eqref{newStateq} that a minimizers is strictly positive.
\end{proof}

For the arguments that follow, we only need one side of the stationarity equation.

\begin{definition}
	We say that $m$ satisfies the \textit{stationarity inequality} if
	\begin{equation} \label{eqn:stationarity_inequality}
	- \tilde{\beta} \int \dd v \tilde{\rho}(v)m(v)J(u-v)+I'(m(u)) \leq 0 \qquad \text{for } \tilde{\rho}\lambda \text{-a.e. } u \in  \mathbb{T}^d,
	\end{equation}
	or equivalently,
	\begin{equation*}
	m(u) \leq \tanh\left(\tilde{\beta} \int \dd v \tilde{\rho}(v)m(v)J(u-v)\right)\quad \text{for } \tilde{\rho}\lambda \text{-a.e. } u \in  \mathbb{T}^d.
	\end{equation*}
\end{definition}

We start with a technical result based on the stationarity inequality and convexity that, afterwards, will allow us to establish the uniqueness of non-negative minimizers.

\begin{lemma} \label{lemma:convexity}
	For any two different  (i.e. $\tilde{\rho}\lambda(\{m_1 \neq m_2\})>0$) non-negative solutions $m_1$ and $m_2$ to the stationarity inequality \eqref{eqn:stationarity_inequality} we have 
	\begin{equation*}
	\tilde{\rho}\lambda(\{m_2<m_1\})=0 \quad \Rightarrow \quad I_{\tilde{\rho}}(m_2)<I_{\tilde{\rho}}(m_1). 
	\end{equation*}
\end{lemma}

\begin{proof}
	Assume, that $m_1$ and $m_2$ are two different non-negative solutions to \eqref{eqn:stationarity_inequality} where  $\tilde{\rho}\lambda(\{m_2<m_1\})=0$ (so, $\tilde{\rho}\lambda(\{m_2>m_1\})>0$). Consider the linear interpolation
	\begin{equation*}
	F:[0,1] \times \mathbb{T}^d \rightarrow [0,1] \quad ; \quad F(t,u):=(1-t)m_1(u)+tm_2(u). 
	\end{equation*}
	Then $I_{\tilde{\rho},\tilde{\beta}} F$ is a continuous, differentiable function in $t$, so at any $0\leq s \leq 1$ the Fundamental theorem of calculus gives
	\begin{equation*} 
	I_{\tilde{\rho},\tilde{\beta}}((1-s)m_1+sm_2)-I_{\tilde{\rho}}(m_1)=\int_0^s{(I_{\tilde{\rho},\tilde{\beta}}F)'(t)}dt 
	\end{equation*}
	where by Lemma \ref{lemma:derivative}
	\begin{multline}
	(I_{\tilde{\rho},\tilde{\beta}}F)'(t) \label{HomDer} \\
	= \int \dd u \tilde{\rho}(u) \left(m_2(u)-m_1(u)\right) \left\{ -\tilde{\beta} \int \dd v \tilde{\rho}(v)F(t,v)J(u-v)+I'(F(t,u)) \right\}. 
	\end{multline}

	Inserting $\pm \left((1-t)I'(m_1(u))+tI'(m_2(u))\right)$, we may restate the bracketed expression in \eqref{HomDer} in terms of the stationarity inequality:
	\begin{align*}
	& -\tilde{\beta} \int \dd v \tilde{\rho}(v)F(t,v)J(u-v) +I'(F(t,u)) \\
	& \qquad \qquad = \left(-\tilde{\beta} \int \dd v\tilde{\rho}(v) F(t,v)J(u-v)+(1-t)I'(m_1(u))+tI'(m_2(u))\right) \\
	& \qquad \quad \qquad \quad \; \,  +I'(F(t,u)) -((1-t)I'(m_1(u))+tI'(m_2(u)) \\
	& \qquad \qquad = (1-t) \left(-\tilde{\beta} \int \dd v \tilde{\rho}(v) m_1(v)J(u-v) +I'(m_1(u)) \right) 	\\
	& \qquad \qquad \qquad  \; \, + t \, \left(-\tilde{\beta} \int \dd v \tilde{\rho}(v) m_2(v))J(u-v)+I'(m_2(u)) \right) \\
	& \qquad \qquad \qquad \; \,   +I'(F(t,u)) -((1-t)I'(m_1(u))+tI'(m_2(u)). 
	\end{align*}
	Thus, using that $m_1,m_2$ satisfy the stationarity inequality \eqref{eqn:stationarity_inequality}, we find
	\begin{multline*}
	-\tilde{\beta} \int \dd v \tilde{\rho}(v)F(t,v)J(u-v) +I'(F(t,u)) \\
	 \leq  I'((1-t)m_1(u) + t m_2(u)) -((1-t)I'(m_1(u))+tI'(m_2(u)).
	\end{multline*}
	The function $I'(\cdot) = \arctanh(\cdot)$ is strictly convex on $[0,\infty)$. Using that $\tilde{\rho}\lambda(\{m_2 < m_1\}) = 0$ and $\tilde{\rho}\lambda(\{m_2 > m_1\}) > 0$, we find by \eqref{HomDer} that for all $t \in (0,1)$:
	\begin{equation*}
	(I_{\tilde{\rho},\tilde{\beta}}F)'(t) < 0
	\end{equation*}
	which establishes the claim.
\end{proof}

\begin{lemma} \label{lemma:uniqueness_of_non_neg_minimizer}
	A non-negative global minimizer to the rate-function $I_{\tilde{\rho},\tilde{\beta}}$ is unique.
\end{lemma}

\begin{proof}
	Assume that there are two different non-negative global minimizers $m_1$ and $m_2$ for $I_{\tilde{\rho},\tilde{\beta}}$. Then both of them satisfy the stationarity equation \eqref{SCzzz} and, hence, the stationarity inequality \eqref{eqn:stationarity_inequality}. Since $m_1$ and $m_2$ are different, at least one of them is different to the point-wise maximum $m_1 \lor m_2$, and both $m_1$ and $m_2$ are bounded from above by $m_1 \lor m_2$. Therefore, showing, that $m_1 \lor m_2$ satisfies the stationarity inequality \eqref{eqn:stationarity_inequality}, one may apply the Lemma \ref{lemma:convexity} to get the contradiction
	\begin{equation*}
	I_{\tilde{\rho},\tilde{\beta}}(m_1 \lor m_2)<I_{\tilde{\rho},\tilde{\beta}}(m_1) \lor I_{\tilde{\rho},\tilde{\beta}}(m_2)=I_{\tilde{\rho},\tilde{\beta}}(m_1) \wedge I_{\tilde{\rho},\tilde{\beta}}(m_2).
	\end{equation*}
	By non-negativity of $J$ and $\tilde{\rho}$ and the application of the stationarity inequality \eqref{eqn:stationarity_inequality} to $m_1$ and $m_2$, we have at $\tilde{\rho}\lambda$-a.e. $u \in \mathbb{T}^d$:
	\begin{align*}
	&\tanh\left(\tilde{\beta}((\tilde{\rho}(m_1 \lor m_2))*J)(u)\right) \geq \tanh\left(\tilde{\beta}((\tilde{\rho}m_1)*J)(u)\right) \geq m_1(u),\\
	&\tanh\left(\tilde{\beta}((\tilde{\rho}(m_1 \lor m_2))*J)(u)\right) \geq \tanh\left(\tilde{\beta}((\tilde{\rho}m_2)*J)(u)\right) \geq m_2(u),
	\end{align*}
	so,
	\begin{equation*}
	\tanh\left(\tilde{\beta}((\tilde{\rho}(m_1 \lor m_2))*J)(u)\right) \geq (m_1 \lor m_2)(u)  \quad \tilde{\rho}\lambda\text{-a.s.}
	\end{equation*}
	which concludes the proof.
\end{proof}

A combination of the results above allows us to identify the minimizers of $I_{\tilde{\rho},\tilde{\beta}}$.

\begin{proposition}[characterization of the set of minimizers to $I_{\tilde{\rho}}$] \label{proposition:characterization_of_minimizers}
	The rate-function $I_{\tilde{\rho},\tilde{\beta}}$ is either uniquely minimized by the trivial profile $m \equiv 0$, or there are exactly two minimizers $m_1$ and $m_2$ where
	\begin{itemize}
		\item[a)] $m_1$ is strictly positive and
		\item[b)] $m_2=-m_1$.
	\end{itemize} 
\end{proposition}

\begin{proof}
	Since $I_{\tilde{\rho},\tilde{\beta}}$ has compact level sets, there is at least one global minimizer $m$. Suppose this particular minimizer $m$ is not equal to $0$. Then by Lemma \ref{lemma:minimizers_are_nonnegative} (b), there are two global minimizers $m_1,m_2$ with $m_1 \geq 0$ and $m_1 = - m_2$ and in addition $m \in \{m_1,m_2\}$. By (c), $m_1$ is an interior point of $\text{Ball}(L^\infty)$, which implies by Lemma \ref{lemma:strictly_positive_minimizers} that $m_1$ is strictly positive.
	
	\smallskip
	
	The uniqueness of a non-negative profile was established in Lemma \ref{lemma:uniqueness_of_non_neg_minimizer}.
\end{proof}

\subsection{Convergence of profiles, proof of Proposition \ref{proposition:minimizers_I_and_convergence}} \label{section:Ising_convergence_of_profiles}

For proof of the law of Proposition \ref{proposition:minimizers_I_and_convergence}, we will critically exploit symmetry properties of the Potts model. We start with a short discussion of establishing the permutation invariance of the profiles that are obtained from the $r$-state Potts model under the maps $\pi_\Lambda$.

\smallskip

Suppose our fuzzy class has $r$ elements. Let $s$ be an element from the permutation group of $\{1,\dots,r\}$. Also denote by $s$ the map induced by $s$ on $\{1,\dots,r\}^\Lambda$ by acting coordinate-wise. In turn, we can define a map $\hat{s} : \mathcal{P}(\bT^d \times \{1,\dots,r\}) \rightarrow \mathcal{P}(\bT^d \times \{1,\dots,r\})$, by $(\hat{s}\pi)(u,i) = \pi(u,s(i))$. Recall the definition of $\pi_{\Lambda}$ in \eqref{def:pi_Lambda}. By construction, we have the following permutation property: $\pi_{\Lambda} \circ s = \hat{s}\circ \pi_{\Lambda}$. Also, by definition of the Potts model, we have $\mu_{\Lambda,\beta,r} \circ s^{-1} = \mu_{\Lambda,\beta,r}$.


Combining these statements, we obtain 
\begin{equation*}
P_{\Lambda,\beta,r} \circ \hat{s}^{-1} = \mu_{\Lambda,\beta,r} \circ \pi_{\Lambda}^{-1} \circ \hat{s}^{-1} =\mu_{\Lambda,\beta,r} \circ s^{-1} \circ \pi_{\Lambda}^{-1} = \mu_{\Lambda,\beta,r} \circ \pi_{\Lambda}^{-1} = P_{\Lambda,\beta,r},
\end{equation*}
i.e. the permutation symmetry of the Potts model carries over to that of the density profiles.

\begin{proof}[Proof of Proposition \ref{proposition:minimizers_I_and_convergence}]
	The set
	\begin{equation*}
	U := \left\{Q \in \cP(\cP(\bT^d \times \{1,2\})) \, \middle| \, \text{for all permutations } s: \, Q \circ \hat{s}^{-1} = Q  \right\}
	\end{equation*}
	is closed for the weak topology, due to the continuity of the maps $s$. It follows that any limit point of the measures $P_n$ must be in $U$ as well.
	
	By Lemma \ref{lemma:possible_limit_points_with_ldp_upper_bound}, it follows that limiting points must be concentrated on the minimizers of the large deviation rate function, which have been identified in Proposition \ref{proposition:characterization_of_minimizers}. 
	
	In both setting (a) and (b), the intersection of these two sets contains only one element, proving the result.
\end{proof}

\subsection{The limiting kernel, proof of Theorem \ref{theorem:form_of_kernels_in_Gibbs_case}} \label{section:Ising_form_of_kernel}

To establish the result of Theorem \ref{theorem:form_of_kernels_in_Gibbs_case}, it suffices by Lemma \ref{lemma:convergence_of_kernels} to find a unique limiting measure for each fuzzy class. In the proof of Theorem 2.7 in \cite{JaKu17}, it has been shown for fuzzy classes $l$ of size at least $3$ that if the profile $\nu$ is good then the limiting measure is the spatially homogeneous equi-distribution. Integrating this measure against the function $\cA_{r_l}$ yields one of the factors in the limiting kernel. Thus, it suffices to give an explicit formula for the integral of $\cA_2$ for the Ising classes.

\begin{lemma}
	Let $\beta \geq 0$ and let $\rho$ be a density profile on $\bT^d$ with $\rho \lambda(\bT^d) > 0$. Denote $\tilde{\beta} = \beta \rho\lambda(\bT^d)$ and $\tilde{\rho} = (\rho\lambda(\bT^d))^{-1} \rho$. Suppose that $\phi^{*,+}$ is the unique minimizer of $I_{\tilde{\rho},\tilde{\beta}}$ with non-negative local magnetization $m_{\tilde{\rho}}$. Set $P \in \cP(\bT^d \times \{1,2\})$ by
	\begin{equation*}
	P = \frac{1}{2}\left(\delta_{\phi^{*,+}} + \delta_{\phi^{*,-}}  \right).
	\end{equation*}
	Note that if the $I_{\tilde{\rho},\tilde{\beta}}$ has a unique global minimizer $\phi^* = \phi^{*,+} = \phi^{*,-}$ with $0$ local magnetization, then the formula for $P$ remains valid also.
	
	Then we have that
	\begin{equation*}
	\int \cA_2(\pi,u,\tilde{\beta}) P(\dd \pi) = 2\exp\left\{\tilde{\beta} \int \dd v \, \tilde{\rho}(v)J(u-v)\right\} \frac{1}{\sqrt{1-m_{\tilde{\rho}}(u)^2}}.
	\end{equation*}
\end{lemma}

\begin{proof}
	Recall, that the magnetization profiles $\pm m_{\tilde{\rho}}$ correspond
	to color profiles having densities
	\begin{equation*}
	\frac{\dd \pi[1]}{\dd \tilde{\rho}\lambda}(\cdot) =\frac{1 \pm m_{\tilde{\rho}}(\cdot)}{2} \quad \text{and} \quad 
	\frac{\dd \pi[2]}{\dd \tilde{\rho}\lambda}(\cdot) =\frac{1 \mp m_{\tilde{\rho}}(\cdot)}{2}.
	\end{equation*}
	We thus obtain
	\begin{align}
	& \int \cA_2(\pi,u,\tilde{\beta}) P(\dd \pi) \notag \\
	& \qquad = \int \sum_{i=1}^2 \exp\left\{2\tilde{\beta}(J*\pi[i])(u)\right\} \left(\frac{1}{2}\delta_{\phi^{*,+}}+\frac{1}{2}\delta_{\phi^{*,-}}\right)(\dd \pi) \notag \\
	& \qquad = 2\left(\frac{1}{2}\exp\left\{ 2 \tilde{\beta} \int \frac{1+m_{\tilde{\rho}}(v)}{2}J(u-v) \tilde{\rho}(v)\dd v\right\} \right. \notag \\
	& \hspace{3.5cm} \qquad \quad  \left. + \frac{1}{2} \exp \left\{ 2\tilde{\beta} \int \frac{1-m_{\tilde{\rho}}(v)}{2}J(u-v)\tilde{\rho}(v)\dd v \right\}\right) \notag \\
	& \qquad = \exp\left\{\tilde{\beta} \int \tilde{\rho}(v)J(u-v)\dd v\right\} \left(\exp\left\{\tilde{\beta} \int \tilde{\rho}(v) m_{\tilde{\rho}}(v)J(u-v)dv\right\} \right.  \label{eqn:kernel_integral_ising} \\
	& \hspace{3.5cm} \qquad \quad + \left. \exp\left\{-\tilde{\beta} \int \tilde{\rho}(v)m_{\tilde{\rho}}(v)J(u-v) \dd v\right\} \right). \notag
	\end{align}
	We first simplify the term involving two exponentials with $m_{\tilde{\rho}}$ factors. By application of the stationarity equation \eqref{newStateq}
	\begin{equation*}
	\tilde{\beta} \int \tilde{\rho}(v) m_{\tilde{\rho}}(v)J(u-v) \dd v = \arctanh(m_{\tilde{\rho}}(u)) = \frac{1}{2}\log\left(\frac{1+m_{\tilde{\rho}}(u)}{1-m_{\tilde{\rho}}(u)}\right),
	\end{equation*}
	the braced expression further becomes
	\begin{multline*}
	\exp\left\{\tilde{\beta}\int \tilde{\rho}(v) m_{\tilde{\rho}}(v)J(u-v)\dd v\right\}+\exp\left\{-\tilde{\beta} \int \tilde{\rho}(v)m_{\tilde{\rho}}(v)J(u-v) \dd v\right\} \\
	 = \sqrt{\frac{1+m_{\tilde{\rho}}(u)}{1-m_{\tilde{\rho}}(u)}}+\sqrt{\frac{1-m_{\tilde{\rho}}(u)}{1+m_{\tilde{\rho}}(u)}} = \frac{2}{\sqrt{1-m_{\tilde{\rho}}(u)^2}}. 
	\end{multline*}
	Inserting $\tilde{\beta}=\beta\rho\lambda(\mathbb{T}^d)$ and $\tilde{\rho}=\frac{\rho}{\rho\lambda(\mathbb{T}^d)}$ into the first factor of \eqref{eqn:kernel_integral_ising}, we get the result
	\begin{equation*}
	\int \cA_2(\pi,u,\tilde{\beta}) P(\dd \pi)
	=2\exp\left\{\beta \int \dd v \rho_k(v)J(u-v)\right\} \frac{1}{\sqrt{1-m_{\tilde{\rho}}(u)^2}}.
	\end{equation*}
\end{proof}

\appendix

\section{Absolute continuity of limiting measures} \label{appendix:absolute_continuity}

We start by establishing the result of Lemma \ref{lemma:possible_limits_of_profiles}.

\begin{lemma} 
	For each $n$ let $\Lambda_n \subseteq \bT^d_n$. Set 
	\begin{equation*}
	\nu_n := \frac{1}{|\bT^n_d|} \sum_{i \in \Lambda_n} \delta_{i}.
	\end{equation*}
	Then any weak limit point $\nu$ of the sequence $\nu_n$ can be considered as an element in the unit ball of $L^\infty(\bT^d, \dd x)$.
\end{lemma}

\begin{proof}
	Let $\nu$ be a weak limit point of the sequence $\nu_n$. Without loss of generality, we can assume $\nu_n \rightarrow \nu$ weakly. 
	
	First, we prove that $\nu \in L^1(\bT^d,\dd x)$. As the set $L^1(\bT^d, \dd x)$ can be interpreted as a subset of $\cM(\bT^d)$, it suffices to show that for any element $f \in L^\infty(\bT^d,\dd x)$ we have that $\ip{f}{\nu} \leq  \vn{f}_{L^\infty(\bT^d,\dd x)}$ as this implies $\vn{\nu}_{L^1(\bT^d,\dd x)} \leq 1$. 
	
	Fix $\varepsilon > 0$. By Lusin's theorem, cf. \cite[Theorem 7.1.13]{Bo07}, there is a continuous function $f_\varepsilon$ with $\nu(f \neq f_\varepsilon) \leq \varepsilon$ and $\vn{f_\varepsilon} \leq \vn{f}$. Thus
	\begin{equation*}
	\left| \int f \dd \nu \right| \leq 2 \vn{f} \varepsilon + \left| \int f_\varepsilon \dd \nu \right|.
	\end{equation*}
	The right hand side is approximated by $\ip{f_\varepsilon}{\nu_n}$, which by assumption is bounded by $\vn{f}$ uniformly in $n$. Sending $\varepsilon \downarrow 0$, we find 
	\begin{equation*}
	\left| \int f \dd \nu \right| \leq \vn{f}
	\end{equation*}
	establishing that $\nu \in L^1(\bT^d,\dd x)$. 
	
	\smallskip
	
	Denote by $g$ the density of $\nu$ with respect to the Lebesgue measure. We prove it is bounded by using the same trick as above, but by paring with functions $f \in L^1(\bT^d,\dd x)$. First, recall that the space of bounded continuous functions is dense in $L^1(\bT^d,\nu)$, thus, we can approximate $f$ in $L^1$ by continuous bounded functions $f_k$. We find
	\begin{equation*}
	\left|\int fg \dd x\right| = \lim_k | \int f_k g  \dd x |= \lim_k \lim_n \ip{f_k}{\nu_n}.
	\end{equation*}
	For each fixed $k$, the function $f_k$ is uniformly continuous. As the measures $\nu_n$ are constructed by putting Dirac-masses at a grid with distance $n^{-1}$ between neighboring points, we can bound the integrals $|\ip{f_k}{\nu_n}| \leq |\int f_k \dd x| + \delta_k(n)$, where $\lim_n \delta_k(n) = 0$. It follows that $|\int fg \dd x | \leq \vn{f}_{L^1(\bT^d,\dd x)}$ implying that $g \in L^\infty(\bT^d,\dd x)$ and $\vn{g}_{L^\infty(\bT^d,\dd x)} \leq 1$.
	
\end{proof}

\section{Continuity of limiting functionals} \label{appendix:continuity_limiting_functionals}

The following lemma is an adaptation from Lemma 2.1 in \cite{HRZ15}, taking into account that for each $n$, we consider a different space.

\begin{lemma} \label{lemma:continuity_of_kernel_abstract}
	Let $\{\cX_n\}_{n \geq 1}$ be a sequence of spaces and let $(\cX,d_\cX)$ and $(\cY,d_\cY)$ be two metric spaces. Suppose there are maps $\eta_n : \cX_n \rightarrow \cX$ in such a way that for every $x \in \cX$ there are $x_n \in \cX_n$ with $\eta_n(x_n) \rightarrow x$.
	
	Let $f_n : \cX_n \rightarrow \cY$, $f : \cX \rightarrow \cY$ and suppose that for all sequences $x_n \in \cX_n$ with $\eta_n(x_n) \rightarrow x$, we have $f_n(x_n) \rightarrow f(x)$. Then $f$ is continuous.
\end{lemma}

\begin{proof}
	As $\cX$ is metric, it suffices to prove for any sequence $x_n \in \cX$ converging to $x \in \cX$ that $f(x_n) \rightarrow f(x)$.
	
	For every $n$ let $x_{n,k} \in \cX_k$ be such that $\eta_k(x_{n,k}) \rightarrow x_n$. By assumption, we find $f_k(x_{n,k}) \rightarrow f(x_n)$. Using that $\cX$ and $\cY$ are metric, we can choose $k(n)$ large enough such that
	\begin{equation} \label{eqn:lemma_continuous_limit}
	d_{\cX}(\eta_{k(n)}(x_{n,k(n)}),x_n) \leq \frac{1}{n}, \qquad d_{\cY}(f_{k(n)}(x_{n,k(n)}),f(x_n)) \leq \frac{1}{n}.
	\end{equation}
	The first inequality, combined with $\eta_n(x_n) \rightarrow x$, implies that $\lim_n \eta_{k(n)}(x_{n,k(n)}) = x$. By assumption, we find that $\lim_n f_{k(n)}(x_{n,k(n)} = f(x)$. Therefore, using the second inequality in \eqref{eqn:lemma_continuous_limit} 
	\begin{equation*}
	\lim_n d_{\cY}(f(x_n)),f(x)) \leq \lim_n d_{\cY}(f(x_n),f_{k(n)}(x_{n,k(n)})) + d_{\cY}(f_{k(n)}(x_{n,k(n)}),f(x)) = 0.
	\end{equation*}
	We conclude that $f$ is continuous.
	
\end{proof}

Use for $\cX_n$ is the space of configurations at the finite $n$ level. $\cY$ is the space of single site kernels. Let $f_n$ be as in \eqref{Representation_Kernel}

The function $f$, we will define directly, based on the minimizers of the rate function. I.e. as in \cite{JaKu17} directly in the single minimum case, average over two minimizers in the Ising case.

\section{Limits when the LDP has multiple minimizers} \label{appendix:limits_LDP_more_minimizers}

The following result is folklore, for which the authors could not find a good reference. 

\begin{lemma} \label{lemma:possible_limit_points_with_ldp_upper_bound}
Let $(\mu_n)_{n \geq 1}$ be a sequence of probability-measures on a Polish space $X$ satisfying an large deviation upper bound at speed $\{r(n)\}_{n \geq 1}$ with good rate function $I$. Further assume, that the rate-function only possesses a finite number of minimizers $m_1,\ldots, m_s$. 

Then, for any weakly convergent subsequence \[ \mu_{n_k} \stackrel{k \rightarrow \infty}{\rightarrow} \mu \in  \mathcal{M}_1(X, \mathcal{B}(\mathcal{X})) \] there are $\l_1, \ldots ,\l_s \geq 0$ with $ \sum_{i=1}^s{\l_i}=1$ such that $\mu=\sum_{i=1}^s{\l_i \delta_{m_i}}$
\end{lemma}

\begin{proof}
Assume, that $ \mu \notin \{\sum_{i=1}^s{\l_i\delta_{m_i}} \, : \, \l_1, \ldots, \l_s \geq 0 \text{ and } \sum_{i=1}^s{\l_i}=1   \}$. Then we have $\mu(\{m_1, \ldots, m_s \}^c)>0$. By inner regularity of the probability measure $\mu$ (e.g. \cite[Theorem 7.1.7]{Bo07}) this still holds true for some compact subset $K \subseteq \{m_1, \ldots, m_s \}^c$. Moreover, employing the fact, that $X$ is metrizable, we can find some open set $O \subseteq X$ where 
\begin{equation*}
K \subseteq O \subseteq \overline{O} \subseteq \{m_1, \ldots, m_s \}^c.
\end{equation*}
The large deviation upper bound, using that $\overline{O}$ is a closed set not containing the minimizers $\{m_1,\dots,m_s\}$ of $I$, yields
\begin{equation*}
\limsup_n \frac{1}{r(n)} \log \mu_n(\overline{O}) \leq - \inf_{x \in \overline{O}} I(x) < 0
\end{equation*}
and, hence,
\begin{equation}\label{eqn:LLN_upper_bound}
\lim_n \mu_n(\overline{O}) = 0.
\end{equation}
On the other hand, using the lower bound in the Portmanteau theorem, we find
\begin{equation*} 
\liminf_{k \rightarrow \infty}{\mu_{n_k}(\overline{O})} \geq \liminf_{k \rightarrow \infty}{\mu_{n_k}(O)} \geq \mu(O) \geq \mu(K)>0, 
\end{equation*}
which contradicts \eqref{eqn:LLN_upper_bound}. 
\end{proof}


\bibliographystyle{alpha} 
\bibliography{KraaijBib}{}

\begin{thebibliography}{dHRvZ15}

\bibitem[Bog07]{Bo07}
Vladimir~I. Bogachev.
\newblock {\em Measure Theory}.
\newblock Springer-Verlag, 2007.

\bibitem[dHRvZ15]{HRZ15}
Frank den Hollander, Frank Redig, and Willem van Zuijlen.
\newblock {Gibbs}-non-{Gibbs} dynamical transitions for mean-field interacting
  {Brownian} motions.
\newblock {\em Stochastic Processes and their Applications}, 125(1):371 -- 400,
  2015.

\bibitem[EFS93]{EFS93}
Aernout~C.D. Enter, Roberto Fern\'{a}ndez, and Alan~D. Sokal.
\newblock Regularity properties and pathologies of position-space
  renormalization-group transformations: Scope and limitations of {Gibbsian}
  theory.
\newblock {\em Journal of Statistical Physics}, 72(5-6):879--1167, 1993.

\bibitem[EK10]{EK10}
Victor Ermolaev and Christof K\"{u}lske.
\newblock Low-temperature dynamics of the {C}urie-{W}eiss model: Periodic
  orbits, multiple histories, and loss of {G}ibbsianness.
\newblock {\em Journal of Statistical Physics}, 141(5):727--756, 2010.

\bibitem[FdHM13]{FHM13}
R.~Fern\'{a}ndez, F.~den Hollander, and J.~Mart\'{i}nez.
\newblock Variational description of {G}ibbs-non-{G}ibbs dynamical transitions
  for the {C}urie-{W}eiss model.
\newblock {\em Communications in Mathematical Physics}, 319(3):703--730, 2013.

\bibitem[FdHM14]{FHM14}
R.~Fern\'{a}ndez, F.~den Hollander, and J.~Mart\'{i}nez.
\newblock Variational description of {G}ibbs-non-{G}ibbs dynamical transitions
  for spin-flip systems with a {K}ac-type interaction.
\newblock {\em Journal of Statistical Physics}, 156(2):203--220, 2014.

\bibitem[Geo11]{Ge11}
Hans-Otto Georgii.
\newblock {\em Gibbs measures and phase transitions}, volume~9 of {\em de
  Gruyter Studies in Mathematics}.
\newblock Walter de Gruyter \& Co., Berlin, second edition, 2011.

\bibitem[HK04]{HK04}
O.~H{\"a}ggstr{\"o}m and C.~K{\"u}lske.
\newblock Gibbs properties of the fuzzy {P}otts model on trees and in mean
  field.
\newblock {\em Markov Process. Related Fields}, 10(3):477--506, 2004.

\bibitem[JK16]{JaKu16}
Benedikt Jahnel and Christof K\"{u}lske.
\newblock The {Widom-Rowlinson} model under spin flip: Immediate loss and sharp
  recovery of quasilocality.
\newblock {\em To appear in Annals of Applied Probability}, 2016.

\bibitem[JK17a]{JaKu17a}
Benedikt Jahnel and Christof K\"{u}lske.
\newblock Gibbsian representation for point processes via hyperedge potentials.
\newblock {\em arXiv preprint arXiv:1707.05991}, 2017.

\bibitem[JK17b]{JaKu17}
Benedikt Jahnel and Christof K{\"u}lske.
\newblock Sharp thresholds for {Gibbs}-non-{Gibbs} transitions in the fuzzy
  {Potts} model with a {Kac}-type interaction.
\newblock {\em Bernoulli}, 23(4A):2808--2827, 11 2017.

\bibitem[Kle14]{Kl14}
Achim Klenke.
\newblock {\em Probability theory}.
\newblock Universitext. Springer, London, second edition, 2014.
\newblock A comprehensive course.

\bibitem[KLN07]{KN07}
Christof K\"{u}lske and Arnaud Le~Ny.
\newblock Spin-flip dynamics of the {C}urie-{W}eiss model: Loss of
  {G}ibbsianness with possibly broken symmetry.
\newblock {\em Communications in Mathematical Physics}, 271(2):431--454, 2007.

\bibitem[Rue99]{Ru99}
David Ruelle.
\newblock {\em Statistical mechanics: Rigorous results}.
\newblock World Scientific, 1999.

\end{thebibliography}

\end{document}